\documentclass{amsart}

\newtheorem{theorem}{Theorem}[section]
\newtheorem{lemma}[theorem]{Lemma}
\newtheorem{corollary}[theorem]{Corollary}
\theoremstyle{definition}

\newtheorem{example}[theorem]{Example}

\theoremstyle{remark}

\numberwithin{equation}{section}

\begin{document}

\title{Hyperelliptic curves over $\mathbb{F}_q$ and Gaussian hypergeometric series}


\author{Rupam Barman}
\address{Department of Mathematics, Indian Institute of Technology, Hauz Khas, New Delhi-110016, INDIA}
\curraddr{}
\email{rupam@maths.iitd.ac.in}
\thanks{}

\author{Gautam Kalita}
\address{Department of Mathematics, Darrang College, Tezpur-784001, Sonitpur, Assam, INDIA}
\curraddr{}
\email{haigautam@gmail.com}
\thanks{}

\date{19th November, 2013}

\begin{abstract}
Let $d\geq2$ be an integer. Denote by $E_d$ and $E'_{d}$ the hyperelliptic curves over $\mathbb{F}_q$ given by
$$E_d: y^2=x^d+ax+b~~~ \text{and} ~~~E'_d: y^2=x^d+ax^{d-1}+b,$$ respectively.
We explicitly find the number of $\mathbb{F}_q$-points on $E_d$ and $E'_d$ in
terms of special values of ${_{d}}F_{d-1}$ and ${_{d-1}}F_{d-2}$ Gaussian hypergeometric series with characters
of orders $d-1$, $d$, $2(d-1)$, $2d$, and $2d(d-1)$ as parameters. This gives a solution to a problem posed by Ken Ono \cite[p. 204]{ono2} on
special values of ${_{n+1}}F_n$ Gaussian hypergeometric series for $n > 2$.
We also show that the results of Lennon \cite{lennon1} and the authors \cite{BK3} on trace of Frobenius of elliptic curves follow
from the main results.
\end{abstract}
\maketitle
\section{Introduction and statement of results}
 In \cite{greene}, Greene introduced the notion of hypergeometric functions
over finite fields or \emph{Gaussian hypergeometric series} which are analogous to the classical hypergeometric series. Since then,
many interesting relations between special values of these functions and the number of $\mathbb{F}_p$-points on certain varieties
have been obtained. For example, Koike \cite{koike}, Fuselier \cite{fuselier}, Lennon \cite{lennon1, lennon2}, and the authors
\cite{BK3, BK5} gave formulas for the number of $\mathbb{F}_q$-points on elliptic curves in terms of special values of ${_{2}}F_1$ Gaussian
hypergeometric series. Ono \cite{ono1} expressed the trace of Frobenius of the Clausen family of elliptic curves in terms of
a ${_{3}}F_2$ Gaussian hypergeometric series. 
Also, Vega \cite{vega} and the authors \cite{BK1, BK2} studied this problem for certain families of more general algebraic curves.
\par In all the known results connecting Gaussian hypergeometric series and algebraic curves,
expressions are obtained in terms of ${_{2}}F_1$ and ${_{3}}F_2$ Gaussian hypergeometric series.
Hence, the task remained to find similar results for ${_{n+1}}F_n$ Gaussian hypergeometric series with $n\geq3$.
Ahlgren and Ono studied this problem and deduced the value of a ${_{4}}F_3$ hypergeometric series at $1$ over $\mathbb{F}_p$ in
terms of representations of $4p$ as a sum of four squares using the fact that
the Calabi-Yau threefold is modular \cite{Ahlgren}.  For $n>3$, the non-trivial values of ${_{n+1}}F_n$
Gaussian hypergeometric series have been difficult to obtain, and this problem was also mentioned by Ono \cite[p. 204]{ono2}.
Recently, the authors \cite{BK4} found expressions for the number of zeros of the polynomial $x^d+ax+b$ over a finite field $\mathbb{F}_q$
in terms of special values of ${_{d}}F_{d-1}$ and ${_{d-1}}F_{d-2}$ Gaussian hypergeometric series with characters of orders $d-1$ and $d$
as parameters under the condition that $q\equiv 1$ $($mod $d(d-1))$, where $d\geq2$. The first author with Saikia \cite{BS1} obtained similar formulas
for the number of zeros of more general polynomials $x^d+ax^i+b$ and $x^d+ax^{d-i}+b$ over $\mathbb{F}_q$  in terms of
special values of $_{\frac{d-i}{i}}F_{\frac{d-2i}{i}}$ and $_{\frac{d}{i}}F_{\frac{d-i}{i}}$ Gaussian hypergeometric series with characters of orders
$\frac{d}{i}-1$ and $\frac{d}{i}$ as parameters under the condition that $q\equiv 1$ $($mod $\frac{d(d-i)}{i^{2}})$, where $d\geq 2$, $0 < i <d$,
and $i | d$.
\par Hyperelliptic curves play a crucial role in cryptography and counting points on such curves over a finite field is an important problem.
No connection between Gaussian hypergeometric series and hyperelliptic curves have been shown to date, except for elliptic curves which are
special cases of hyperelliptic curves. In this paper we consider the problem of expressing the number of points on certain families of hyperelliptic curves
in terms of special values of hypergeometric functions over finite fields. We also show that the results of Lennon \cite{lennon1} and
the authors \cite{BK3} on trace of Frobenius of elliptic curves follow from the main results.
\par We begin with some preliminary definitions needed to state our results. Let $q=p^e$ be a power of an odd prime $p$, and let $\mathbb{F}_q$
be the finite field of $q$ elements. Let $\widehat{\mathbb{F}_q^\times}$
denote the group of multiplicative characters $\chi$ on $\mathbb{F}_q^{\times}$.
We extend each character $\chi \in \widehat{\mathbb{F}_q^{\times}}$ to all of $\mathbb{F}_q$ by
setting $\chi(0):=0$. If $A$ and $B$ are two characters on $\mathbb{F}_q^{\times}$, then ${A \choose B}$ is defined by
\begin{align}\label{eq0}
{A \choose B}:=\frac{B(-1)}{q}J(A,\overline{B})=\frac{B(-1)}{q}\sum_{x \in \mathbb{F}_q}A(x)\overline{B}(1-x),
\end{align}
where $J(A, B)$ denotes the usual Jacobi sum and $\overline{B}$ is the inverse of $B$.
\par
Recall the definition of the Gaussian hypergeometric series over $\mathbb{F}_q$ first defined by Greene in \cite{greene}.
Let $n$ be a positive integer. For characters $A_0, A_1,\ldots, A_n$ and $B_1, B_2,\ldots, B_n$ on $\mathbb{F}_q$,
the Gaussian hypergeometric series ${_{n+1}}F_n$ is defined to be
\begin{align}\label{eq00}
{_{n+1}}F_n\left(\begin{array}{cccc}
                A_0, & A_1, & \ldots, & A_n\\
                 & B_1, & \ldots, & B_n
              \end{array}\mid x \right):=\frac{q}{q-1}\sum_{\chi}{A_0\chi \choose \chi}{A_1\chi \choose B_1\chi}
\cdots {A_n\chi \choose B_n\chi}\chi(x),
\end{align}
where the sum is over all characters $\chi$ on $\mathbb{F}_q$.
\par Throughout the paper, for $d\geq 2$ and  $a, b \neq 0$, we consider the hyperelliptic curves $E_d$ and $E'_d$ over $\mathbb{F}_q$ given by
\begin{align} \label{curve-1}
 E_d: y^2=x^d+ax+b
\end{align}
and
\begin{align} \label{curve-2}
 E'_d: y^2=x^d+ax^{d-1}+b,
\end{align}
respectively.
\par In this paper we give two explicit formulas for the number of $\mathbb{F}_q$-points on $E_d$ in terms of
special values of ${_{d}}F_{d-1}$ and ${_{d-1}}F_{d-2}$ Gaussian hypergeometric series with characters
of orders $d$, $2(d-1)$, and $2d(d-1)$ as parameters. In case of $d$ is even, odd powers of a character of degree $2(d-1)$ appear in the 2nd row
of a ${_{d}}F_{d-1}$ Gaussian hypergeometric series; whereas even powers of the same character appear in the 2nd row
of a ${_{d-1}}F_{d-2}$ Gaussian hypergeometric series in case $d$ is odd. The main results are stated below. 
The notation $\phi$ and $\varepsilon$ are reserved for quadratic and trivial
characters on $\mathbb{F}_q$, respectively. We also fix a generator $T$ for the
cyclic group of multiplicative characters $\widehat{\mathbb{F}_q^{\times}}.$
\begin{theorem}\label{thm1}
Let $q=p^e$, $p>0$ be a prime, and let $N_d$ denote the number of $\mathbb{F}_q$-points on $E_d$. If $d \geq 2$ is even
and $q\equiv1$ $($mod $2d(d-1))$, then
\begin{align}
N_d=&q+\phi(b)+q^{\frac{d}{2}}\phi(b(d-1))\notag\\
&\times {_{d}}F_{d-1}\left(\begin{array}{cccccccccc}
                \phi, & \varepsilon, & \chi, &\chi^2, &\ldots, &\chi^\frac{d-2}{2}, &\chi^\frac{d+2}{2}, & \ldots, & \chi^{d-1}\\
                 & \phi, & \psi,  &\psi^3, &\ldots, &\psi^{d-3}, &\psi^{d+1}, & \ldots, & \psi^{2d-3}
              \end{array}\mid \alpha \right),\notag
\end{align}
where $\chi$ and $\psi$ are characters of order $d$ and $2(d-1)$, respectively; and $\alpha = \dfrac{d}{a}\left(\dfrac{bd}{a(d-1)}\right)^{d-1}$.
\end{theorem}
\begin{theorem}\label{thm2}
Let $q=p^e$, $p>0$ be a prime, and let $N_d$ denote the number of $\mathbb{F}_q$-points on $E_d$. If $d \geq 3$ is odd
and $q\equiv1$ $($mod $2d(d-1))$, then
\begin{align}
N_d&=q+\phi(b)-\phi(-b(d-1))T^{\frac{(q-1)}{4}}(-1)\notag\\
&+q^{\frac{d-1}{2}}\phi(-b(d-1))T^{\frac{(q-1)}{4}}(-1)T^{\frac{q-1}{2(d-1)}}\left(-\dfrac{1}{\alpha}\right)\times\notag\\
&{_{d-1}}F_{d-2}\left(\begin{array}{cccccccccc}
                \xi^{d-2}, & \xi^{3d-4}, &\ldots, &\xi^{d^2-3d+1}, &\xi^{d^2-d-1}, & \ldots, & \xi^{2d^2-5d+2}\\
                 & \psi^2,  &\ldots, & \psi^{d-3}, &\psi^{d-1}, & \ldots, & \psi^{2d-4}
              \end{array}\mid -\alpha\right),\notag
\end{align}
where $\psi$ and $\xi$ are characters of order $2(d-1)$ and $2d(d-1)$, respectively; and $\alpha=\dfrac{d}{a}\left(\dfrac{db}{a(d-1)}\right)^{d-1}$.
\end{theorem}
We now state similar results for the curve $E'_d: y^2=x^d+ax^{d-1}+b$. If $d$ is odd, then we obtain a formula for the number of $\mathbb{F}_q$-points
on $E'_d$ under the condition that $q\equiv1$ $($mod $d(d-1))$.
\begin{theorem}\label{thm3}
Let $q=p^e$, $p>0$ be a prime, and let $N'_d$ denote the number of $\mathbb{F}_q$-points on $E'_d$.
If $d\geq 2$ is even and $q\equiv1$ $($mod $2d(d-1))$, then
\begin{align}
N'_d=&q+\phi(b)+q^{\frac{d}{2}}\phi((d-1))\times\notag\\
&{_{d}}F_{d-1}\left(\begin{array}{cccccccccc}
                \phi, & \varepsilon, & \chi, &\chi^2, &\ldots, &\chi^\frac{d-2}{2}, &\chi^\frac{d+2}{2}, & \ldots, & \chi^{d-1}\\
                 & \phi, & \psi,&\psi^3,  &\ldots, &\psi^{d-3}, &\psi^{d+1}, & \ldots, & \psi^{2d-3}
              \end{array}\mid \beta \right),\notag
\end{align}
where $\chi$ and $\psi$ are characters of order $d$ and $2(d-1)$, respectively; and $\beta=\dfrac{bd^d}{a^d(d-1)^{d-1}}$.
\end{theorem}
\begin{theorem}\label{thm4}
Let $q=p^e$, $p>0$ be a prime, and let $N'_d$ denote the number of $\mathbb{F}_q$-points on $E'_d$.
If $d\geq 3$ is odd and $q\equiv1$ $($mod $d(d-1))$, then
\begin{align}
N'_d=&q+q^{\frac{d-1}{2}}\phi(-ad)\times\notag\\
&{_{d-1}}F_{d-2}\left(\begin{array}{cccccccccc}
                \eta,& \eta^3 &\eta^5, &\ldots, &\eta^{d-2}, &\eta^{d+2}, & \ldots, & \eta^{2d-3}, & \eta^{2d-1}\\
                 & \rho,  &\rho^2, &\ldots & \rho^{\frac{d-3}{2}}, &\rho^{\frac{d+1}{2}}, & \ldots, & \rho^{d-2}, &\varepsilon
              \end{array}\mid -\beta \right),\notag
\end{align}
where $\eta$ and $\rho$ are characters of order $2d$ and $(d-1)$, respectively; and $\beta=\dfrac{bd^d}{a^d(d-1)^{d-1}}$.
\end{theorem}
From Theorem \ref{thm1} and Theoem \ref{thm3}, we have the following result.
\begin{corollary}
 Let $d\geq 2$ be even and $a, b \in \mathbb{F}_q^{\times}$. For $q\equiv 1 (\text{mod}~2d(d-1))$, the hyperelliptic curves $y^2=x^d+ax+b$ 
 and $y^2=x^d+ax^{d-1}+b$ have equal number of $\mathbb{F}_q$-points if $b^{d-2}=1$ and $b$ is a quadratic residue in $\mathbb{F}_q$.
\end{corollary}
We now give two examples to show how the above theorems are applied in specific values of $d$.
\begin{example}
 Let $d=4$. Let $a, b \in \mathbb{F}_q^{\times}$ and $q\equiv 1 (\text{mod}~24)$. Also,
 let  $\chi_4$ and $\chi_6$ be characters of order $4$ and $6$, respectively. Then from Theorem \ref{thm1}, we deduce that
 \begin{align}
  &\#\{(x, y)\in \mathbb{F}_q^{2}: y^2=x^{4}+ax+b\}\notag\\
  &=q+\phi(b)+q^2\phi(3b)\cdot {_{4}}F_{3}\left(\begin{array}{cccccccccc}
                \phi, & \varepsilon, & \chi_4, & \chi_4^{3}\\
                 &\phi, & \chi_6, &\chi_6^5
              \end{array}\mid \frac{256b^3}{27a^4}\right).\notag
 \end{align}
From Theorem \ref{thm3}, we deduce that
\begin{align}
  &\#\{(x, y)\in \mathbb{F}_q^{2}: y^2=x^{4}+ax^3+b\}\notag\\
  &=q+\phi(b)+q^2\phi(3)\cdot {_{4}}F_{3}\left(\begin{array}{cccccccccc}
                \phi, & \varepsilon, & \chi_4, & \chi_4^{3}\\
                 &\phi, & \chi_6, &\chi_6^5
              \end{array}\mid \frac{256b}{27a^4}\right).\notag
 \end{align}
\end{example}
\begin{example}
 Let $d=5$ and $a, b \in \mathbb{F}_q^{\times}$. Then from Theorem \ref{thm2}, for $q\equiv1 (\text{mod}~40)$, we deduce that
 \begin{align}
  &\#\{(x, y)\in \mathbb{F}_q^{2}: y^2=x^{5}+ax+b\}\notag\\
  &=q+\phi(b)-\phi(-b)T^{\frac{q-1}{4}}(-1)+q^2\phi(-b)T^{\frac{q-1}{4}}(-1)T^{\frac{q-1}{8}}\left(-\frac{256a^5}{3125b^4}\right)\times\notag\\
  &\hspace{.6cm} {_{4}}F_{3}\left(\begin{array}{cccccccccc}
                \xi^3, & \xi^{11}, & \xi^{19}, & \xi^{27}\\
                 &\psi^2, & \psi^4, &\psi^6
              \end{array}\mid -\frac{3125b^4}{256a^5}\right),\notag
 \end{align}
 where $\xi$ and $\psi$ are characters of order $8$ and $40$, respectively.\\
Again from Theorem \ref{thm4}, for $q\equiv1 (\text{mod}~20)$, we deduce that
\begin{align}
  &\#\{(x, y)\in \mathbb{F}_q^{2}: y^2=x^{5}+ax^4+b\}\notag\\
  &=q+q^2\phi(-5a)\cdot {_{4}}F_{3}\left(\begin{array}{cccccccccc}
                \eta, & \eta^3, & \eta^7, & \eta^9\\
                 &\chi_4, & \chi_4^3, &\varepsilon
              \end{array}\mid -\frac{3125b}{256a^5}\right),\notag
 \end{align}
 where $\chi_4$ and $\eta$ are characters of order $4$ and $10$, respectively.
\end{example}
\section{Preliminaries}
In this section, we recall some results which we will use to prove our main results.
We start defining the additive character $\theta: \mathbb{F}_q \rightarrow \mathbb{C}^{\times}$ by
\begin{align}
\theta(\alpha)=\zeta^{\text{tr}(\alpha)}\notag
\end{align}
where $\zeta=e^{2\pi i/p}$ and $\text{tr}: \mathbb{F}_q \rightarrow \mathbb{F}_p$ is the trace map given by
$$\text{tr}(\alpha)=\alpha + \alpha^p + \alpha^{p^2}+ \cdots + \alpha^{p^{e-1}}.$$
For $A\in \widehat{\mathbb{F}_q^\times}$, the \emph{Gauss sum} is defined by
\begin{align}
G(A):=\sum_{x\in \mathbb{F}_q}A(x)\zeta^{\text{tr}(x)}=\sum_{x\in \mathbb{F}_q}A(x)\theta(x).\notag
\end{align}
Recall that $T$ denotes a fixed generator of the cyclic group $\widehat{\mathbb{F}_q^\times}$. We denote by $G_m$ the Gauss sum $G(T^m)$.
The following lemma provides a formula for the multiplicative inverse of a Gauss sum.
\begin{lemma}\emph{(\cite[Eqn. 1.12]{greene}).}\label{lemma1}
If $k\in\mathbb{Z}$ and $T^k\neq\varepsilon$, then
$$G_kG_{-k}=qT^k(-1).$$
\end{lemma}
The following lemma gives a relationship between Gauss and Jacobi sums.
\begin{lemma}\emph{(\cite[Eqn. 1.14]{greene}).}\label{lemma2}
If $T^{m-n}\neq\varepsilon$, then $$G_mG_{-n}=q{T^m\choose
T^n}G_{m-n}T^n(-1)=J(T^m, T^{-n}) G_{m-n}.$$
\end{lemma}
\par We now state the \emph{orthogonality relations} for multiplicative characters.
\begin{lemma}\emph{(\cite[Chapter 8]{ireland}).}\label{lemma3}
We have
\begin{enumerate}
\item $\displaystyle\sum_{x\in\mathbb{F}_q}T^n(x)=\left\{
                                  \begin{array}{ll}
                                    q-1 & \hbox{if~ $T^n=\varepsilon$;} \\
                                    0 & \hbox{if ~~$T^n\neq\varepsilon$.}
                                  \end{array}
                                \right.$
\item $\displaystyle\sum_{n=0}^{q-2}T^n(x)~~=\left\{
                            \begin{array}{ll}
                              q-1 & \hbox{if~~ $x=1$;} \\
                              0 & \hbox{if ~~$x\neq1$.}
                            \end{array}
                          \right.$
\end{enumerate}
\end{lemma}
Using orthogonality, we can write $\theta$ in terms of Gauss sums as given in the following lemma.
\begin{lemma}\emph{(\cite[Lemma 2.2]{fuselier}).}\label{lemma4}
For all $\alpha \in \mathbb{F}_q^{\times}$, $$\theta(\alpha)=\frac{1}{q-1}\sum_{m=0}^{q-2}G_{-m}T^m(\alpha).$$
\end{lemma}
\begin{theorem}\emph{(\cite[Davenport-Hasse Relation]{lang}).}\label{lemma5}
Let $m$ be a positive integer and let $q=p^e$ be a prime power such that $q\equiv 1 (\text{mod}~m)$.
For multiplicative characters $\chi, \psi \in \widehat{\mathbb{F}_q^\times}$, we have
\begin{align}
\prod_{\chi^m=1}G(\chi \psi)=-G(\psi^m)\psi(m^{-m})\prod_{\chi^m=1}G(\chi).
\end{align}
\end{theorem}
We have the following two special cases of Davenport-Hasse relation. For details, see \cite{BK4}.
\begin{corollary}\label{coro1}
Let $d$ be a positive integer, $l\in\mathbb{Z}$, $q=p^e\equiv1$
$($mod $d)$, and $t\in\{1, -1\}$.
\begin{itemize}
  \item If $d> 1$ is odd, then
  \begin{align}\label{eq11}
G_lG_{l+t\frac{q-1}{d}}G_{l+t \frac{2(q-1)}{d}}\cdots G_{l+t
\frac{(d-1)(q-1)}{d}}=q^{\frac{d-1}{2}}T^{\frac{(d-1)(d+1)(q-1)}{8d}}(-1)T^{-l}\left(d^d\right)G_{ld}.
\end{align}
  \item If $d$ is even, then
\begin{align}\label{eq12}
G_lG_{l+t \frac{q-1}{d}}G_{l +t \frac{2(q-1)}{d}}\cdots G_{l+t
\frac{(d-1)(q-1)}{d}}=q^{\frac{d-2}{2}}G_{\frac{q-1}{2}}T^{\frac{(d-2)(q-1)}{8}}(-1)T^{-l}(d^d)G_{ld}.
\end{align}
\end{itemize}
\end{corollary}
\section{Proof of Theorem \ref{thm1} and Theorem \ref{thm2}}
For $d\geq 2$, the hyperelliptic curve $E_d$ defined over $\mathbb{F}_q$ is given by $$E_d: y^2=x^d+ax+b.$$
Then the number of points on $E_d$ over $\mathbb{F}_q$ is given by $$N_d=\#\{(x, y)\in\mathbb{F}_q^2: P(x,y)=0\},$$
where $$P(x,y)=x^d+ax+b-y^2.$$
Using the elementary identity from \cite{ireland}
\begin{align}
\sum_{z\in \mathbb{F}_q}\theta(zP(x,y))=\left\{
                                            \begin{array}{ll}
                                              q & \hbox{if $P(x,y)=0$;} \\
                                              0 & \hbox{if $P(x,y)\neq 0,$}
                                            \end{array}
                                          \right.\notag
\end{align}
we obtain that
\begin{align}\label{eq1}
q\cdot N_d&=\sum_{x,y,z \in\mathbb{F}_q}\theta(zP(x,y))\notag\\
=~& q^2+\sum_{z\in\mathbb{F}_q^\times}\theta(zb)+\sum_{y,z\in\mathbb{F}_q^\times}\theta(zb)\theta(-zy^2)+
\sum_{x,z\in\mathbb{F}_q^\times}\theta(zb)\theta(zx^d)\theta(zax)\notag\\
&+\sum_{x,y,z\in\mathbb{F}_q^\times}\theta(zb)\theta(zx^d)\theta(zax)\theta(-zy^2)\notag\\
:=~& q^2+A+B+C+D.
\end{align}
We use Lemma \ref{lemma4} and Lemma \ref{lemma3} repeatedly in each term of \eqref{eq1} to simplify and express in terms of Gauss sums. We obtain
\begin{align}
A=\frac{1}{q-1}\sum_{z\in\mathbb{F}_q^\times}\sum_{l=0}^{q-2}G_{-l}T^l(zb)
=\frac{1}{q-1}\sum_{l=0}^{q-2}G_{-l}T^l(b)\sum_{z\in\mathbb{F}_q^\times}T^l(z)=-1.\notag
\end{align}
Similarly,
\begin{align}
B=&\frac{1}{(q-1)^2}\sum_{l,m=0}^{q-2}G_{-l}G_{-m}T^l(b)T^m(-1)\sum_{z\in\mathbb{F}_q^\times}T^{l+m}(z)\sum_{y\in\mathbb{F}_q^\times}T^{2m}(y).\notag
\end{align}
This term is nonzero only if $m=0$ or $\frac{q-1}{2}$ and $l=-m$. Thus the fact $G_0=-1$ and Lemma \ref{lemma1} yield
\begin{align}
B=1+q\phi(b).\notag
\end{align}
Expanding the next term, we have
\begin{align}
C=\frac{1}{(q-1)^3}\sum_{l,m,n=0}^{q-2}G_{-l}G_{-m}G_{-n}T^l(b)T^n(a)\sum_{z\in \mathbb{F}_q^\times}T^{l+m+n}(z)
\sum_{x\in \mathbb{F}_q^\times}T^{md+n}(x).\nonumber
\end{align}
Finally,
\begin{align}
D=&\frac{1}{(q-1)^4}\sum_{l,m,n,k=0}^{q-2}G_{-l}G_{-m}G_{-n}G_{-k}T^l(b)T^n(a)T^k(-1) \times \nonumber\\
&\sum_{z\in \mathbb{F}_q^\times}T^{l+m+n+k}(z)\sum_{x\in \mathbb{F}_q^\times}T^{md+n}(x)\sum_{y\in\mathbb{F}_q^\times}T^{2k}(y)\nonumber.
\end{align}
The innermost sum of $D$ is nonzero only if $k=0$ or $\frac{q-1}{2}$. For $k=0$, we obtain the term equal to $-C$. Further for $k=\frac{q-1}{2}$,
we denote the term by $D_{\frac{q-1}{2}}$ given as
\begin{align}
D_{\frac{q-1}{2}}=&~\frac{\phi(-1)G_{\frac{q-1}{2}}}{(q-1)^3}\sum_{l,m,n=0}^{q-2}G_{-l}G_{-m}G_{-n}T^l(b)T^n(a)\times\notag\\
&\sum_{z\in \mathbb{F}_q^\times}T^{l+m+n+\frac{q-1}{2}}(z)\sum_{x\in \mathbb{F}_q^\times}T^{md+n}(x).\notag
\end{align}
Here, the term $D_{\frac{q-1}{2}}$ is zero unless $n=-md$ and $l=(d-1)\{m+\frac{q-1}{2(d-1)}\}$. Then using Lemma \ref{lemma3}, we have
\begin{align}\label{eq2}
D_{\frac{q-1}{2}}=\frac{\phi(-b)G_{\frac{q-1}{2}}}{(q-1)}\sum_{m=0}^{q-2}G_{(d-1)(-m-\frac{q-1}{2(d-1)})}G_{-m}G_{dm}T^m\left(\frac{b^{d-1}}{a^d}\right).
\end{align}
Thus, combining values of $A,B,C,D$ all together in \eqref{eq1}, we have
\begin{align}\label{eq3}
q\cdot N_d=q^2+q\phi(b)+D_{\frac{q-1}{2}}.
\end{align}
If $d\geq 2$ is an even integer, from \eqref{eq11} and \eqref{eq12}, we have 
\begin{align}
G_{dm}=\frac{G_mG_{m+ \frac{q-1}{d}}G_{m + \frac{2(q-1)}{d}}\cdots G_{m+
\frac{(d-1)(q-1)}{d}}}{q^{\frac{d-2}{2}}G_{\frac{q-1}{2}}T^{\frac{(d-2)(q-1)}{8}}(-1)T^{-m}(d^d)}\notag
\end{align}
and
\begin{align}
G_{(d-1)(-m-\frac{q-1}{2(d-1)})}=\frac{G_{-m-\frac{q-1}{2(d-1)}}G_{-m-\frac{3(q-1)}{2(d-1)}}G_{-m-\frac{5(q-1)}{2(d-1)}}
\cdots G_{-m-\frac{(2d-3)(q-1)}{2(d-1)}}}{q^{\frac{d-2}{2}}T^{\frac{(d-2)d(q-1)}{8(d-1)}}(-1)T^{m+\frac{q-1}{2(d-1)}}\left((d-1)^{d-1}\right)}.\notag
\end{align}
Using these in \eqref{eq2}, we obtain
\begin{align}
D_{\frac{q-1}{2}}=&\frac{\phi(-b(d-1))T^{\frac{(d-2)(q-1)}{8(d-1)}}(-1)}{q^{d-2}(q-1)}
\sum_{m=0}^{q-2}\{G_mG_{-m}\}\{G_{m+ \frac{q-1}{d}}G_{-m-\frac{q-1}{2(d-1)}}\}\notag\\
&\times \{G_{m+ \frac{2(q-1)}{d}}G_{-m-\frac{3(q-1)}{2(d-1)}}\}\cdots
\{G_{m+ \frac{(d-2)(q-1)}{2d}}G_{-m-\frac{(d-3)(q-1)}{2(d-1)}}\} \notag\\
&\times \{G_{m+ \frac{d(q-1)}{2d}}G_{-m-\frac{(d-1)(q-1)}{2(d-1)}}\}
\{G_{m+ \frac{(d+2)(q-1)}{2d}}G_{-m-\frac{(d+1)(q-1)}{2(d-1)}}\}\notag\\
&\times \cdots\{G_{m+ \frac{(d-1)(q-1)}{d}}G_{-m-\frac{(2d-3)(q-1)}{2(d-1)}}\}T^m\left(\frac{d^db^{d-1}}{(d-1)^{d-1}a^d}\right)\notag\\
=&\frac{\phi(-b(d-1))T^{\frac{(d-2)(q-1)}{8(d-1)}}(-1)}{q^{d-2}(q-1)}\sum_{m=0}^{q-2}\{G_{m+ \frac{(q-1)}{2}}G_{-m}\}
\{G_mG_{-m-\frac{(q-1)}{2}}\}\notag\\
&\times \{G_{m+ \frac{q-1}{d}}G_{-m-\frac{q-1}{2(d-1)}}\}\{G_{m+ \frac{2(q-1)}{d}}G_{-m-\frac{3(q-1)}{2(d-1)}}\}\notag\\
&\times \cdots  \{G_{m+ \frac{(d-2)(q-1)}{2d}}G_{-m-\frac{(d-3)(q-1)}{2(d-1)}}\}\{G_{m+ \frac{(d+2)(q-1)}{2d}}G_{-m-\frac{(d+1)(q-1)}{2(d-1)}}\}\notag\\
&\times\cdots\{G_{m+\frac{(d-1)(q-1)}{d}} G_{-m-\frac{(2d-3)(q-1)}{2(d-1)}}\}T^m\left(\frac{d^db^{d-1}}{(d-1)^{d-1}a^d}\right).\notag
\end{align}
Using Lemma \ref{lemma2} in each term in bracket, we deduce that
\begin{align}
D_{\frac{q-1}{2}}=&\frac{q^2\phi(-b(d-1))T^{\frac{(d-2)(q-1)}{8(d-1)}}(-1)}{(q-1)}\{G_{\frac{q-1}{2}}G_{-\frac{q-1}{2}}\}
\{G_{\frac{(d-2)(q-1)}{2d(d-1)}} G_{-\frac{(d-2)(q-1)}{2d(d-1)}}\} \notag\\
&\times \cdots \{G_{\frac{2(q-1)}{2d(d-1)}} G_{-\frac{2(q-1)}{2d(d-1)}}\}
\sum_{m=0}^{q-2}{T^{m+\frac{q-1}{2}} \choose T^m}{T^{m} \choose T^{m+\frac{q-1}{2}}}{T^{m+\frac{q-1}{d}} \choose T^{m+\frac{q-1}{2(d-1)}}}\notag\\
&\times{T^{m+\frac{2(q-1)}{d}} \choose T^{m+\frac{3(q-1)}{2(d-1)}}}\cdots {T^{m+\frac{(d-2)(q-1)}{2d}} \choose T^{m+\frac{(d-3)(q-1)}{2(d-1)}}}{T^{m+\frac{(d+2)(q-1)}{2d}} \choose T^{m+\frac{(d+1)(q-1)}{2(d-1)}}}
\cdots {T^{m+\frac{(d-1)(q-1)}{d}} \choose T^{m+\frac{(2d-3)(q-1)}{2(d-1)}}} \notag\\
&\times T^m\left(\frac{d^db^{d-1}}{(d-1)^{d-1}a^d}\right).\notag
\end{align}
Thus Lemma \ref{lemma1} yields
\begin{align}
D_{\frac{q-1}{2}}&=q^{\frac{d+2}{2}}\phi(b(d-1))\notag\\
&\times {_{d}}F_{d-1}\left(\begin{array}{cccccccccc}
                \phi, & \varepsilon, & \chi, & \chi^2, &\ldots, &\chi^\frac{d-2}{2}, &\chi^\frac{d+2}{2}, & \ldots, & \chi^{d-1}\\
                 & \phi, & \psi, &\psi^3, &\ldots, &\psi^{d-3}, &\psi^{d+1}, & \ldots, & \psi^{2d-3}
              \end{array}\mid \alpha \right),\notag
\end{align}
where $\alpha = \dfrac{d}{a}\left(\dfrac{bd}{a(d-1)}\right)^{d-1}$.
Hence \eqref{eq3} completes the proof of Theorem \ref{thm1}.
\par Now we prove Theorem \ref{thm2}.
For $d>1$ odd integer, the Davenport-Hasse relations \eqref{eq11} and \eqref{eq12} yield
\begin{align}
G_{dm}=\frac{G_mG_{m+ \frac{q-1}{d}}G_{m+ \frac{2(q-1)}{d}}\cdots G_{m+
\frac{(d-1)(q-1)}{d}}}{q^{\frac{d-1}{2}}T^{\frac{(d-1)(d+1)(q-1)}{8d}}(-1)T^{-m}\left(d^d\right)},\notag
\end{align}
and
\begin{align}
G_{(d-1)(-m-\frac{q-1}{2(d-1)})}=\frac{G_{-m-\frac{q-1}{2(d-1)}}G_{-m-\frac{3(q-1)}{2(d-1)}}G_{-m-\frac{5(q-1)}{2(d-1)}}
\cdots G_{-m-\frac{(2d-3)(q-1)}{2(d-1)}}}{q^{\frac{d-3}{2}}G_{\frac{q-1}{2}}T^{\frac{(d-3)(q-1)}{8}}(-1)T^{m+\frac{q-1}{2(d-1)}}\left((d-1)^{d-1}\right)}.
\notag
\end{align}
We use these relations in \eqref{eq2} to obtain
\begin{align}
D_{\frac{q-1}{2}}=&\frac{\phi(-b(d-1))T^{\frac{(3d-1)(q-1)}{8d}}(-1)}{q^{d-2}(q-1)}\sum_{m=0}^{q-2}\{G_mG_{-m}\}
\{G_{m+ \frac{q-1}{d}}G_{-m-\frac{q-1}{2(d-1)}}\}\notag\\
&\times \{G_{m+ \frac{2(q-1)}{d}}G_{-m-\frac{3(q-1)}{2(d-1)}}\}\cdots
\{G_{m+ \frac{(d-1)(q-1)}{2d}}G_{-m-\frac{(d-2)(q-1)}{2(d-1)}}\}\notag\\
&\times \{G_{m+ \frac{(d+1)(q-1)}{2d}}G_{-m-\frac{d(q-1)}{2(d-1)}}\}
\{G_{m+ \frac{(d+3)(q-1)}{2d}}G_{-m-\frac{(d+2)(q-1)}{2(d-1)}}\}\notag\\
&\times \cdots\{G_{m+ \frac{(d-1)(q-1)}{d}}G_{-m-\frac{(2d-3)(q-1)}{2(d-1)}}\}T^m\left(\frac{d^db^{d-1}}{(d-1)^{d-1}a^d}\right).\notag
\end{align}
To eliminate $G_mG_{-m}$, we use the facts that if $m\neq0$, then $G_mG_{-m}=qT^m(-1)$; and if $m=0$, then $G_mG_{-m}=1=qT^m(-1)-(q-1)$
in appropriate identities above. After that, we rearrange the second term to deduce that
\begin{align}
D_{\frac{q-1}{2}}=&\displaystyle\frac{\phi(-b(d-1))T^{\frac{(3d-1)(q-1)}{8d}}(-1)}{q^{d-3}(q-1)}
\sum_{m=0}^{q-2}\{G_{m+ \frac{q-1}{d}}G_{-m-\frac{q-1}{2(d-1)}}\}\times\notag\\
& \{G_{m+ \frac{2(q-1)}{d}}G_{-m-\frac{3(q-1)}{2(d-1)}}\}\cdots
\{G_{m+ \frac{(d-1)(q-1)}{2d}}G_{-m-\frac{(d-2)(q-1)}{2(d-1)}}\}\times\notag\\
& \{G_{m+ \frac{(d+1)(q-1)}{2d}}G_{-m-\frac{d(q-1)}{2(d-1)}}\}
\{G_{m+ \frac{(d+3)(q-1)}{2d}}G_{-m-\frac{(d+2)(q-1)}{2(d-1)}}\}\times\notag\\
&\ldots\{G_{m+ \frac{(d-1)(q-1)}{d}}G_{-m-\frac{(2d-3)(q-1)}{2(d-1)}}\}T^m\left(-\frac{d^db^{d-1}}{(d-1)^{d-1}a^d}\right)\notag\\
&-\frac{\phi(-b(d-1))T^{\frac{(3d-1)(q-1)}{8d}}(-1)}{q^{d-2}}\{G_{\frac{q-1}{d}}G_{\frac{(d-1)(q-1)}{d}}\}\cdots\times\notag\\
&\{G_{\frac{(d-1)(q-1)}{2d}}G_{\frac{(d+1)(q-1)}{2d}}\}\{G_{-\frac{q-1}{2(d-1)}}G_{-\frac{(2d-3)(q-1)}{2(d-1)}}\}
\cdots \{G_{-\frac{(d-2)(q-1)}{2(d-1)}}G_{-\frac{d(q-1)}{2(d-1)}}\}.\notag
\end{align}
Using Lemma \ref{lemma2} in the first term, and then Lemma \ref{lemma1} in both terms, we obtain
\begin{align}
D_{\frac{q-1}{2}}=&\frac{q^2\phi(-b(d-1))T^{\frac{(3d-1)(q-1)}{8d}}(-1)}{(q-1)}\{G_{\frac{(d-2)(q-1)}{2d(d-1)}}G_{-\frac{(d-2)(q-1)}{2d(d-1)}}\}\times\notag\\ &\{G_{\frac{(d-4)(q-1)}{2d(d-1)}}G_{-\frac{(d-4)(q-1)}{2d(d-1)}}\}
\cdots\{G_{\frac{(q-1)}{2d(d-1)}}G_{-\frac{(q-1)}{2d(d-1)}}\} \sum_{m=0}^{q-2}{T^{m+ \frac{q-1}{d}}\choose T^{m+\frac{q-1}{2(d-1)}}}\times\notag\\
&{T^{m+ \frac{2(q-1)}{d}}\choose T^{m+\frac{3(q-1)}{2(d-1)}}}\cdots{T^{m+ \frac{(d-1)(q-1)}{2d}}\choose T^{m+\frac{(d-2)(q-1)}{2(d-1)}}}{T^{m+ \frac{(d+1)(q-1)}{2d}}\choose T^{m+\frac{d(q-1)}{2(d-1)}}}
\cdots{T^{m+\frac{(d-1)(q-1)}{d}}\choose T^{m+\frac{(2d-3)(q-1)}{2(d-1)}}}\notag\\
&T^m\left(\frac{-d^db^{d-1}}{(d-1)^{d-1}a^d}\right)-\frac{\phi(-b(d-1))T^{\frac{(3d-1)(q-1)}{8d}}(-1)}{q^{d-2}}\{G_{\frac{q-1}{d}}G_{\frac{(d-1)(q-1)}{d}}\} \cdots \times\notag\\
&\{G_{\frac{(d-1)(q-1)}{2d}}G_{\frac{(d+1)(q-1)}{2d}}\}\{G_{-\frac{q-1}{2(d-1)}}G_{-\frac{(2d-3)(q-1)}{2(d-1)}}\} \cdots\{G_{-\frac{(d-2)(q-1)}{2(d-1)}}G_{-\frac{d(q-1)}{2(d-1)}}\}\notag\\
=&\frac{q^{\frac{d+3}{2}}\phi(-b(d-1))T^{\frac{(q-1)}{4}}(-1)}{(q-1)}\sum_{m=0}^{q-2}{T^{m+ \frac{q-1}{d}}\choose T^{m+\frac{q-1}{2(d-1)}}}{T^{m+ \frac{2(q-1)}{d}}\choose T^{m+\frac{3(q-1)}{2(d-1)}}}\cdots\times\notag\\
&{T^{m+ \frac{(d-1)(q-1)}{2d}}\choose T^{m+\frac{(d-2)(q-1)}{2(d-1)}}}{T^{m+ \frac{(d+1)(q-1)}{2d}}\choose T^{m+\frac{d(q-1)}{2(d-1)}}}\cdots{T^{m+\frac{(d-1)(q-1)}{d}}\choose T^{m+\frac{(2d-3)(q-1)}{2(d-1)}}}\notag\\
&T^m\left(-\frac{d^db^{d-1}}{(d-1)^{d-1}a^d}\right)-q\phi(-b(d-1))T^{\frac{(q-1)}{4}}(-1).\notag
\end{align}
Replacing $m$ by $m-\frac{q-1}{2(d-1)}$ in the first term, we have
\begin{align}
D_{\frac{q-1}{2}}=&\frac{q^{\frac{d+3}{2}}\phi(-b(d-1))T^{\frac{(q-1)}{4}}(-1)}{(q-1)}
\sum_{m=0}^{q-2}{T^{m+ \frac{(d-2)(q-1)}{2d(d-1)}}\choose T^{m}}{T^{m+ \frac{(3d-4)(q-1)}{2d(d-1)}}\choose T^{m+\frac{2(q-1)}{2(d-1)}}}
\cdots\times\notag\\
&{T^{m+ \frac{(d^2-3d+1)(q-1)}{2d(d-1)}}\choose T^{m+\frac{(d-3)(q-1)}{2(d-1)}}}{T^{m+ \frac{(d^2-d-1)(q-1)}{2d(d-1)}}\choose T^{m+\frac{(d-1)(q-1)}{2(d-1)}}}
{T^{m+\frac{(2d^2-5d+2)(q-1)}{2d(d-1)}}\choose T^{m+\frac{(2d-4)(q-1)}{2(d-1)}}}\times\notag\\
&T^m\left(-\frac{d^db^{d-1}}{(d-1)^{d-1}a^d}\right)T^{\frac{q-1}{2(d-1)}}\left(-\frac{(d-1)^{d-1}a^d}{d^db^{d-1}}\right)
-q\phi(-b(d-1))T^{\frac{(q-1)}{4}}(-1)\notag\\
=&-q\phi(-b(d-1))T^{\frac{(q-1)}{4}}(-1)+q^{\frac{d+1}{2}}\phi(-b(d-1))
T^{\frac{(q-1)}{4}}(-1)T^{\frac{q-1}{2(d-1)}}(-\frac{1}{\alpha})\times\notag\\
&{_{d-1}}F_{d-2}\left(\begin{array}{cccccccccc}
                \xi^{d-2},& \xi^{3d-4}, &\ldots, &\xi^{d^2-3d+1}, &\xi^{d^2-d-1}, & \ldots, & \xi^{2d^2-5d+2}\\
                 & \psi^2,  &\ldots & \psi^{d-3}, &\psi^{d-1}, & \ldots, & \psi^{2d-4}
              \end{array}\mid -\alpha \right),\notag
\end{align}
where $\alpha = \dfrac{d}{a}\left(\dfrac{bd}{a(d-1)}\right)^{d-1}$.
We complete the proof of Theorem \ref{thm2} by putting the above value of $D_{\frac{q-1}{2}}$ in \eqref{eq3}.
\par Now we show that Theorem 2.1 of Lennon \cite{lennon1} follows from Theorem 1.2.
\begin{theorem}\emph{(\cite[Thm. 2.1]{lennon1}).}\label{lennon-1}
Let $q = p^e, p > 3$ a prime and $q \equiv 1$ $($mod $12)$. Let $E: y^2=x^3+ax+b$ be an
elliptic curve over $\mathbb{F}_q$ with $j(E) \neq 0, 1728$. Then the trace of
the Frobenius map on $E$ can be expressed as
$$a(E(\mathbb{F}_q))=-q\cdot T^{\frac{q-1}{4}}\left(\frac{a^3}{27}\right){_{2}}F_1\left(\begin{array}{cccc}
                T^{\frac{q-1}{12}}, & T^{\frac{5(q-1)}{12}}\\
                 & T^{\frac{q-1}{2}}
              \end{array}\mid -\frac{27b^2}{4a^3} \right).$$
\end{theorem}
\begin{proof}
We have $E_3: y^2=x^3+ax+b$ with $a, b\neq 0$. Hence, $j(E_3)\neq 0, 1728$. For $q\equiv 1$ $($mod $12)$, from Theorem \ref{thm2}, we have
\begin{align}
N_3&=\#\{(x, y)\in \mathbb{F}_q^2: y^2=x^3+ax+b\}\notag\\
&=q+\phi(b)-\phi(-2b)T^{\frac{q-1}{4}}(-1)\notag\\
&\hspace{.7cm}+q \phi(-2b)T^{\frac{q-1}{4}}(-1)T^{\frac{q-1}{4}}\left(-\frac{4a^3}{27b^2}\right){_{2}}F_{1}\left(\begin{array}{cccccccccc}
                \xi,& \xi^5\\
                 & \psi^2 \end{array}\mid -\frac{27b^2}{4a^3} \right)\notag\\
&=q+\phi(b)-\phi(2)\phi(b)T^{\frac{q-1}{4}}(-1)\notag\\
&\hspace{.7cm}+qT^{\frac{q-1}{4}}(4b^2)T^{\frac{q-1}{4}}(-1)T^{\frac{q-1}{4}}\left(-\frac{4a^3}{27b^2}\right){_{2}}F_{1}\left(\begin{array}{cccccccccc}
                T^{\frac{q-1}{12}}, & T^{\frac{5(q-1)}{12}}\\
                 & T^{\frac{q-1}{2}} \end{array}\mid -\frac{27b^2}{4a^3} \right).\notag
\end{align}
Since $\phi(2)=T^{\frac{q-1}{4}}(-1)$, therefore $\phi(2)\phi(b)T^{\frac{q-1}{4}}(-1)=\phi(b)$. Hence
\begin{align}
N_3=&q+qT^{\frac{q-1}{4}}\left(\frac{a^3}{27}\right){_{2}}F_1\left(\begin{array}{cccc}
                T^{\frac{q-1}{12}}, & T^{\frac{5(q-1)}{12}}\\
                 & T^{\frac{q-1}{2}}
              \end{array}\mid -\frac{27b^2}{4a^3} \right).\notag
\end{align}
We complete the proof using the fact that $a(E(\mathbb{F}_q))=a(E_3(\mathbb{F}_q))=q-N_3$.
\end{proof}
\section{Proof of Theorem \ref{thm3} and Theorem \ref{thm4}}
We now prove Theorem \ref{thm3} and Theorem \ref{thm4}.
For $d\geq 2$, the hyperelliptic curve $E'_d$ defined over $\mathbb{F}_q$ is given by $$E'_d: y^2=x^d+ax^{d-1}+b,$$
where $a, b\in \mathbb{F}_q^{\times}$.
Then the number of points on $E'_d$ over $\mathbb{F}_q$ is given by $$N'_q=\#E'_d(\mathbb{F}_q)=\#\{(x,y)\in\mathbb{F}_q^2: P'(x,y)=0\},$$
where $$P'(x,y)=x^d+ax^{d-1}+b-y^2.$$
As shown earlier, we have
\begin{align}
q\cdot N'_d=&~q^2+\sum_{z\in\mathbb{F}_q^\times}\theta(zb)+\sum_{y,z\in\mathbb{F}_q^\times}\theta(zb)\theta(-zy^2)+
\sum_{x,z\in\mathbb{F}_q^\times}\theta(zb)\theta(zx^d)\theta(zax^{d-1})\notag\\
&+\sum_{x,y,z\in\mathbb{F}_q^\times}\theta(zb)\theta(zx^d)\theta(zax^{d-1})\theta(-zy^2)\notag\\
&:=q^2+A+B+C+D.\notag
\end{align}
Following the same procedure as followed in the proof of Theorem \ref{thm1}, we obtain
\begin{align}
A=&-1,\notag\\
B=&1+q\phi(b),\notag
\end{align}
and
\begin{align}
D=&-C+D_{\frac{q-1}{2}},\nonumber
\end{align}
where
\begin{align}\label{eq02}
D_{\frac{q-1}{2}}&=\frac{\phi(-1)G_{\frac{q-1}{2}}}{(q-1)^3}\sum_{l,m,n=0}^{q-2}G_{-l}G_{-m}G_{-n}T^l(b)T^n(a)
\sum_{z\in \mathbb{F}_q^\times}T^{l+m+n+\frac{q-1}{2}}(z)\notag\\
&~~\times \sum_{x\in \mathbb{F}_q^\times}T^{md+n(d-1)}(x).
\end{align}
Combining all values of $A, B, C$ and $D$ together, we have
\begin{align}\label{eq03}
q N'_d=q^2+q\phi(b)+D_{\frac{q-1}{2}}.
\end{align}
Now let $d$ be even. Then the term $D_{\frac{q-1}{2}}$ is zero unless $n=-ld$ and $m=(d-1)\{l+\frac{q-1}{2(d-1)}\}$. Then using Lemma \ref{lemma3}, we have
\begin{align}
D_{\frac{q-1}{2}}=\frac{\phi(-1)G_{\frac{q-1}{2}}}{(q-1)}\sum_{l=0}^{q-2}G_{-l}G_{(d-1)(-l-\frac{q-1}{2(d-1)})}G_{ld}T^l\left(\frac{b}{a^d}\right).\notag
\end{align}
Following the proof of Theorem \ref{thm1}, we deduce that
\begin{align}
&D_{\frac{q-1}{2}}=q^{\frac{d+2}{2}}\phi((d-1))\notag\\
&\times {_{d}}F_{d-1}\left(\begin{array}{cccccccccc}
                \phi, & \varepsilon, & \chi, &\chi^2, &\ldots, &\chi^\frac{d-2}{2}, &\chi^\frac{d+2}{2}, & \ldots, & \chi^{d-1}\\
                 & \phi, & \psi, &\psi^3,  &\ldots, &\psi^{d-3}, &\psi^{d+1}, & \ldots, & \psi^{2d-3}
              \end{array}\mid \beta \right),\notag
\end{align}
where $\beta=\dfrac{bd^d}{a^d(d-1)^{d-1}}$. Hence \eqref{eq03} completes the proof of Theorem \ref{thm3}.
\par
For $d>1$ odd integer, the innermost sums of $D_{\frac{q-1}{2}}$ are nonzero only if $m=l(d-1)$ and $n=d(-l-\frac{q-1}{2d})$, at which both are $q-1$.
Therefore, we have
\begin{align}\label{eq04}
D_{\frac{q-1}{2}}=\frac{\phi(-a)G_{\frac{q-1}{2}}}{(q-1)}\sum_{l=0}^{q-2}G_{-l}G_{-(d-1)l}G_{d(l+\frac{q-1}{2d})}T^l\left(\frac{b}{a^d}\right).
\end{align}
The Davenport-Hasse relations \eqref{eq11} and \eqref{eq12} yield
\begin{align}
G_{-(d-1)l}=\frac{G_{-l}G_{-l-\frac{q-1}{d-1}}G_{-l-\frac{2(q-1)}{d-1}}\cdots G_{-l-
\frac{(d-2)(q-1)}{d-1}}}{q^{\frac{d-3}{2}}G_{\frac{q-1}{2}}T^{\frac{(d-3)(q-1)}{8}}(-1)T^{l}\left((d-1)^{d-1}\right)}\notag
\end{align}
and
\begin{align}
G_{d(l+\frac{q-1}{2d})}=\frac{G_{l+\frac{q-1}{2d}}G_{l+\frac{3(q-1)}{2d}}
G_{l+\frac{5(q-1)}{2d}}\cdots G_{l+\frac{(2d-1)(q-1)}{2d}}}{q^{\frac{d-1}{2}}
T^{\frac{(d-1)(d+1)(q-1)}{8d}}(-1)T^{-l-\frac{q-1}{2d}}\left(d^d\right)}.\notag
\end{align}
We use these relations in \eqref{eq04} to obtain
\begin{align}
D_{\frac{q-1}{2}}=&\frac{\phi(-ad)T^{\frac{(3d-1)(q-1)}{8d}}(-1)}{q^{d-2}(q-1)}\sum_{l=0}^{q-2}\{G_{l+\frac{q-1}{2d}}G_{-l}\}
\{G_{l+ \frac{3(q-1)}{2d}}G_{-l-\frac{q-1}{(d-1)}}\}\notag\\
&\times \{G_{l+\frac{5(q-1)}{2d}}G_{-l-\frac{2(q-1)}{(d-1)}}\}\cdots\{G_{l+ \frac{(d-2)(q-1)}{2d}}G_{-l-\frac{(d-3)(q-1)}{2(d-1)}}\}\notag\\
&\times\{G_{l+\frac{d(q-1)}{2d}}G_{-l-\frac{(d-1)(q-1)}{2(d-1)}}\}\{G_{l+ \frac{(d+2)(q-1)}{2d}}G_{-l-\frac{(d+1)(q-1)}{2(d-1)}}\}\notag\\
&\times \cdots\{G_{l+ \frac{(2d-3)(q-1)}{2d}}G_{-l-\frac{(d-2)(q-1)}{(d-1)}}\}
\{G_{l+ \frac{(2d-1)(q-1)}{2d}}G_{-l}\}T^l(\beta),\notag
\end{align}
where $\beta=\frac{d^db}{(d-1)^{d-1}a^d}$. We now eliminate the term $\{G_{l+\frac{d(q-1)}{2d}}G_{-l-\frac{(d-1)(q-1)}{2(d-1)}}\}$, which is equal to $G_{l+\frac{q-1}{2}}G_{-l-\frac{q-1}{2}}$.
We use the fact that
\begin{align}
G_{l+\frac{q-1}{2}}G_{-l-\frac{q-1}{2}}=\left\{
                                          \begin{array}{ll}
                                            qT^{l+\frac{q-1}{2}}(-1), & \hbox{if $l\neq\frac{q-1}{2}$;} \\
                                            qT^{l+\frac{q-1}{2}}(-1)-(q-1), & \hbox{if $l=\frac{q-1}{2}$}
                                          \end{array}
                                        \right.\notag
\end{align}
in appropriate identities above to obtain
\begin{align}
D_{\frac{q-1}{2}}=&\frac{\phi(ad)T^{\frac{(3d-1)(q-1)}{8d}}(-1)}{q^{d-3}(q-1)}\sum_{l=0}^{q-2}\{G_{l+\frac{q-1}{2d}}G_{-l}\}
\{G_{l+ \frac{3(q-1)}{2d}}G_{-l-\frac{q-1}{(d-1)}}\}\notag\\
&\times \{G_{l+\frac{5(q-1)}{2d}}G_{-l-\frac{2(q-1)}{(d-1)}}\}\cdots\{G_{l+ \frac{(d-2)(q-1)}{2d}}G_{-l-\frac{(d-3)(q-1)}{2(d-1)}}\}\notag\\
&\times \{G_{l+ \frac{(d+2)(q-1)}{2d}}G_{-l-\frac{(d+1)(q-1)}{2(d-1)}}\} \cdots\{G_{l+ \frac{(2d-3)(q-1)}{2d}}G_{-l-\frac{(d-2)(q-1)}{(d-1)}}\}\notag\\
&\times \{G_{l+ \frac{(2d-1)(q-1)}{2d}}G_{-l}\}T^l(-\beta)
-\frac{\phi(-b)T^{\frac{(3d-1)(q-1)}{8d}}(-1)}{q^{d-2}}\notag\\
&\times \{G_{\frac{(d+1)(q-1)}{2d}}G_{\frac{q-1}{2}}\}\{G_{\frac{(d+3)(q-1)}{2d}}G_{-\frac{(d+1)(q-1)}{2(d-1)}}\}\notag\\
&\times \{G_{\frac{(d+5)(q-1)}{2d}}G_{-\frac{(d+3)(q-1)}{2(d-1)}}\}
\cdots\{G_{\frac{2(q-1)}{2d}} G_{-\frac{2(q-1)}{2(d-1)}}\}\notag\\
&\times \{G_{-\frac{2(q-1)}{2d}}G_{\frac{2(q-1)}{2(d-1)}}\}\cdots\{G_{\frac{(d-3)(q-1)}{2d}}G_{\frac{(d+1)(q-1)}{2(d-1)}}\}
\{G_{\frac{(d-1)(q-1)}{2d}}G_{\frac{q-1}{2}}\}.\notag
\end{align}
Using Lemma \ref{lemma2} in the first term, and rearranging both terms, we deduce that
\begin{align}
D_{\frac{q-1}{2}}=&\frac{q^2\phi(ad)T^{\frac{(3d-1)(q-1)}{8d}}(-1)}{(q-1)}\{G_{\frac{q-1}{2d}}G_{-\frac{q-1}{2d}}\}\{G_{\frac{(d-3)(q-1)}{2d(d-1)}}
G_{-\frac{(d-3)(q-1)}{2d(d-1)}}\}\notag\\
&\times \{G_{\frac{(d-5)(q-1)}{2d(d-1)}}
G_{-\frac{(d-5)(q-1)}{2d(d-1)}}\}\cdots \{G_{\frac{2(q-1)}{2d(d-1)}}G_{-\frac{2(q-1)}{2d(d-1)}}\}\notag\\
&\times \sum_{l=0}^{q-2}{T^{l+ \frac{q-1}{2d}}\choose T^{l}}{T^{l+ \frac{3(q-1)}{2d}}\choose T^{l+\frac{q-1}{(d-1)}}}
{T^{l+ \frac{5(q-1)}{2d}}\choose T^{l+\frac{2(q-1)}{(d-1)}}}\cdots{T^{l+ \frac{(d-2)(q-1)}{2d}}\choose T^{l+\frac{(d-3)(q-1)}{2(d-1)}}}\notag\\
&\times {T^{m+ \frac{(d+2)(q-1)}{2d}}\choose T^{l+\frac{(d+1)(q-1)}{2(d-1)}}}\cdots
{T^{l+\frac{(2d-3)(q-1)}{2d}}\choose T^{l+\frac{(d-2)(q-1)}{(d-1)}}}{T^{l+ \frac{(2d-1)(q-1)}{2d}}\choose T^{l}}\notag\\
&\times T^l(-\beta)- \frac{\phi(-b)T^{\frac{(3d-1)(q-1)}{8d}}(-1)}{q^{d-2}}
[\{G_{\frac{(d+1)(q-1)}{2d}}G_{\frac{(d-1)(q-1)}{2d}}\} \notag\\
&\times \{G_{\frac{(d+3)(q-1)}{2d}}G_{\frac{(d-3)(q-1)}{2d}}\}\cdots\{G_{\frac{2(q-1)}{2d}} G_{-\frac{2(q-1)}{2d}}\}]
[\{G_{-\frac{(d+1)(q-1)}{2(d-1)}}G_{\frac{(d+1)(q-1)}{2(d-1)}}\}\notag\\
&\times \{G_{-\frac{(d+3)(q-1)}{2(d-1)}}G_{\frac{(d+3)(q-1)}{2(d-1)}}\}\cdots \{G_{-\frac{2(q-1)}{2(d-1)}}G_{\frac{2(q-1)}{2(d-1)}}\}]
\{G_{\frac{q-1}{2}}G_{\frac{q-1}{2}}\}.\notag
\end{align}
Finally, using Lemma \ref{lemma1}, we have
\begin{align}
&D_{\frac{q-1}{2}}=-q\phi(b)+q^{\frac{d+1}{2}}\phi(-ad)\notag\\
&\times {_{d-1}}F_{d-2}\left(\begin{array}{cccccccccc}
                \eta,& \eta^3 &\eta^5, &\ldots, &\eta^{d-2}, &\eta^{d+2}, & \ldots, & \eta^{2d-3}, & \eta^{2d-1}\\
                 & \rho,  &\rho^2, &\ldots & \rho^{\frac{d-3}{2}}, &\rho^{\frac{d+1}{2}}, & \ldots, & \rho^{d-2}, &\varepsilon
              \end{array}\mid -\beta \right).\notag
\end{align}
We complete the proof of Theorem \ref{thm4} by putting the above value of $D_{\frac{q-1}{2}}$ in \eqref{eq03}.
\par Now we show that Theorem 3.1 of authors \cite{BK3} follows from Theorem 1.4.
\begin{theorem}\emph{(\cite[Thm. 3.1]{BK3}).}\label{bk-1}
Let $q = p^e, p > 2$ a prime and $q \equiv 1$ $($mod $6)$. If If $T \in \widehat{\mathbb{F}_q^{\times}}$ is a generator of the character group,
then the trace of the Frobenius on the elliptic curve $E_1: y^2=x^3+ax^2+b, a\neq 0$ is given by
\begin{align}
a_q(E_1)=-qT^{\frac{q-1}{2}}(-3a)~{_{2}}F_1\left(\begin{array}{cccc}
                T^{\frac{q-1}{6}}, & T^{\frac{5(q-1)}{6}}\\
                 & \varepsilon
              \end{array}\mid -\frac{27b}{4a^3} \right),\nonumber
\end{align}
where $\varepsilon$ is the trivial character on $\mathbb{F}_q$.
\end{theorem}
\begin{proof} Since $E_1: y^2=x^3+ax^2+b$ is an elliptic curve, so $b\neq 0$. Hence, both $a$ and $b$ are non-zero. Putting
$d=3$ in Theorem \ref{thm4}, for $q \equiv 1$ $($mod $6)$, we deduce that
\begin{align}
N'_3&=q+q\phi(-3a){_{2}}F_{1}\left(\begin{array}{cccccccccc}
                \eta,& \eta^5\\
                 & \varepsilon
              \end{array}\mid -\frac{27b}{4a^3} \right)\notag\\
              &=q+q\phi(-3a){_{2}}F_{1}\left(\begin{array}{cccccccccc}
                T^{\frac{q-1}{6}}, & T^{\frac{5(q-1)}{6}}\\
                 & \varepsilon
              \end{array}\mid -\frac{27b}{4a^3} \right).\notag
\end{align}
We complete the proof using the fact that $a_q(E_1)=a_q(E'_3)=q-N'_d$.
\end{proof}

\bibliographystyle{amsplain}

\begin{thebibliography}{10}
\bibitem{Ahlgren} S. Ahlgren and K. Ono, {\it A Gaussian hypergeometric series and Apery number congruences},
J. Reine Angew. Math. \textbf{518} (2000), 187--212.

\bibitem{BK1}
R. Barman and G. Kalita, \textit{Hypergeometric functions and a family of algebraic curves}, Ramanujan J. \textbf{28} (2012), no. 2, 175--185.

\bibitem{BK2}
R. Barman and G. Kalita, \textit{Certain values of Gaussian hypergeometric series and a family of algebraic curves},
Int. J. Number Theory \textbf{8} (2012), no. 4, 945--961.

\bibitem{BK3}
R. Barman and G. Kalita, \textit{Hypergeometric functions over $\mathbb{F}_q$ and traces of
Frobenius for elliptic curves}, Proc. Amer. Math. Soc. \textbf{141} (2013), no. 10, 3403--3410.

\bibitem{BK5}
R. Barman and G. Kalita, \textit{Elliptic curves and special values of Gaussian hypergeometric
series}, J. Number Theory \textbf{133} (2013), no. 9, 3099--3111.

\bibitem{BK4}
R. Barman and G. Kalita, \textit{On the polynomial $x^d+ax+b$ over $\mathbb{F}_q$ and Gaussian hypergeometric series},
Int. J. Number Theory \textbf{9} (2013), no. 7, 1753--1763.

\bibitem{BS1}
R. Barman and N. Saikia, \textit{On the polynomials $x^d+ax^i+b$ and $x^d+ax^{d-i}+b$ over $\mathbb{F}_q$ and Gaussian hypergeometric series},
Ramanujan J. (accepted).

\bibitem{fuselier}
J. Fuselier, \textit{Hypergeometric functions over $\mathbb{F}_p$ and relations to elliptic curves and modular forms},
Proc. Amer. Math. Soc. \textbf{138} (2010), 109--123.

\bibitem {greene}
J. Greene, \textit{Hypergeometric functions over finite fields}, Trans. Amer. Math. Soc.
\textbf{301} (1987), no. 1, 77--101.

\bibitem{ireland}
K. Ireland and M. Rosen, \textit{A Classical Introduction to Modarn Number Theory}, 2nd ed., Graduate Texts in Mathematics, vol. 84,
Springer-Verlag, New York, 1990.

\bibitem{koike}
M. Koike, \textit{Hypergeometric series over finite fields and Ap\'{e}ry numbers}, Hiroshima Math. J. \textbf{22} (1992), 461--467.

\bibitem{lang} S. Lang, \textit{Cyclotomic Fields I and II}, Graduate Texts in Mathematics, vol. 121, Springer-Verlag, New York, 1990.

\bibitem{lennon1}
C. Lennon, \textit{Gaussian hypergeometric evaluations of traces of Frobenius for elliptic curves},
Proc. Amer. Math. Soc. \textbf{139} (2011), 1931--1938.

\bibitem{lennon2} C. Lennon, \textit{Trace formulas for Hecke operators, Gaussian hypergeometric functions, and the modularity of a threefold},
J. Number Theory \textbf{131} (2011), no. 12, 2320--2351.

\bibitem{ono1}
K. Ono, \textit{Values of Gaussian hypergeometric series}, Trans. Amer. Math. Soc. \textbf{350} (1998), no. 3, 1205--1223.

\bibitem{ono2}
K. Ono, \textit{Web of modularity: Arithmetic of the coefficients of modular forms and $q$-series}, CBMS No. 102, Amer. Math. Soc., Providence, R. I., 2004.

\bibitem{vega}
M. V. Vega, \textit{Hypergeometric functions over finite fields and their relations to algebraic curves},
Int. J. Number Theory \textbf{7} (2011), no. 8, 2171--2195.
\end{thebibliography}

\end{document}